\documentclass{amsart}
\usepackage{amsmath,amsthm,amssymb,xypic,array}
\usepackage[all]{xy}
\usepackage{hyperref}

\theoremstyle{plain}

\newtheorem{theorem}{Theorem}[section]
\newtheorem{theoremn}{Theorem}
\newtheorem{proposition}[theorem]{Proposition}
\newtheorem{lemma}[theorem]{Lemma}
\newtheorem{corollary}[theorem]{Corollary}
\newtheorem{conjecture}[theorem]{Conjecture}

\newtheorem{remark}[theorem]{Remark}
\newtheorem{example}[theorem]{Example}

\theoremstyle{definition}

\newtheorem{definition}[theorem]{Definition}

\theoremstyle{remark}

\newtheorem{claim}{Claim}

\DeclareMathOperator{\Bl}{Bs}

\DeclareMathOperator{\rk}{rk}

\DeclareMathOperator{\cod}{cod}

\newcommand{\QED}{\ifhmode\unskip\nobreak\fi\quad {\rm Q.E.D.}} 

\newcommand\iso{\cong}

\newcommand{\f}{\varphi}
\newcommand{\A}{\mathcal{A}}

\renewcommand{\H}{\mathcal{H}}

\newcommand{\I}{\mathcal{I}}
\newcommand{\K}{\mathcal{K}}
\renewcommand{\L}{\mathcal{L}}
\newcommand{\M}{\mathcal{M}}
\newcommand{\N}{\mathbb{N}}
\newcommand\Mo[1]{\overline{M}_{0,#1}}
\renewcommand{\O}{\mathcal{O}}
\renewcommand{\P}{\mathbb{P}}

\newcommand{\R}{\mathbb{R}}

\newcommand{\T}{\mathbb{T}}

\newcommand{\rat}{\dasharrow}
\newcommand{\Cal}{}

\begin{document}
\title{On some fibrations of $\Mo{n}$}

\author{Andrea Bruno and
Massimiliano Mella}
\address{ Dipartimento di Matematica \\
Universit\`a Roma Tre \\L.go S.L.Murialdo, 1 \\
 \newline
\indent 00146 Roma Italia}
\address{Dipartimento di Matematica\\
Universit\`a di Ferrara\\
Via Machiavelli 35\\
 \newline
\indent 44100 Ferrara Italia}
\email{bruno@mat.uniroma3.it}
\email{mll@unife.it}

\date{April 2011}
\subjclass{Primary 14H10 ; Secondary 14D22, 14D06}
\keywords{Moduli space of curves,
pointed  rational curves, fiber type morphism, automorphism}
\thanks{Partially supported by Progetto PRIN 2008 ``Geometria
  sulle variet\`a algebriche'' MUR}
\maketitle


\section*{Introduction}
The moduli space $M_{g,n}$ of smooth
n-pointed curves of genus $g$,
and its projective closure, the Deligne-Mumford
compactification $\overline{M}_{g,n}$,
is a classical object of study that reflects
many of the properties of
families of pointed curves. As a matter of fact, the
study of its biregular geometry is of interest
in itself and has become a central theme in various
areas of mathematics.
One of the intriguing aspects of the geometry of
spaces $\overline{M}_{g,n}$ is that they are
endowed with natural forgetful maps
$$ \phi_I: \overline{M}_{g,n} \to \overline{M}_{g,n-|I|},$$ 
indexed by subsets $I \subset \{1, \ldots, n \}$
corresponding to marked points which are forgotten.

In this paper we will focus on the case $g=0$. 
Even if one would expect the geometry of $\Mo{n}$
to be less complicated than the one of 
$\overline{M}_{g,n}$ for positive genus,
this is not always the case.
Consider for instance the so-called {\em
fibration problem}: Gibney, Keel and Morrison
proved in \cite{GKM} the following:

\begin{theoremn}
Let $g \geq 2$ and $n >0$ and let
  $f:\overline{M}_{g,n}\to X$
  be a fibration, i.e. a dominant morphism with connected fibers to a projective variety $X$. 
  Then there exists $I \subset \{1, \ldots, n \}$ and a birational morphism
  $g: \overline{M}_{g,n-|I|}\to X$ such that $f= g \circ \phi_I$ factors through
  the forgetful map $\phi_I$.
\end{theoremn}

If $g>0$ the product  $\phi_I \times  \phi_J : \overline{M}_{g,n} \to \overline{M}_{g,n-|I|} \times \overline{M}_{g,n-|J|}$ 
of two forgetful maps is never surjective because its image maps to the diagonal down to $\overline{M}_{g} \times \overline{M}_{g}$.
If instead $g=0$, since rational curves do not have moduli, products of forgetful maps
provide more fibrations on $\Mo{n}$.
For instance, there is no way to factor via a simple forgetful map the fibration 
$ \phi_{1,2} \times \phi_{3,4}: \overline{M}_{0,6} \to \overline{M}_{0,4} \times \overline{M}_{0,4}$.\noindent

Nevertheless, if we define a forgetful morphism to be any product of forgetful maps,
in this paper we add some further evidence to the following:

\begin{conjecture} Every fibration from 
$\Mo{n}$ factors through a forgetful morphism. 
\end{conjecture}

Such statement is known to hold  for $n \leq 6$, thanks
to results of Farkas-Gibney, \cite {FG} and Keel-McKernan, \cite {KM} .

This paper continues, along the way of \cite {BM}, the approach to
this problem via the study of the projective geometry of $\Mo{n}$,
made possible by Kapranov's description of $\Mo{n}$.
Kapranov showed in \cite{Ka} that $\Mo{n}$ is identified with the
closure of the subscheme of the Hilbert scheme parametrizing
rational normal curves passing through
$n$  points in linearly general position in $\P^{n-2}$.
Via this identification, any tautological line bundle $\psi_i$
induces a birational morphism $f_i:\Mo{n} \to \P^{n-3}$
which is an iterated blow-up of $\P^{n-3}$ at the strict
transforms of all the
linear spaces spanned by subsets of $n-1$ general points in order
of increasing dimension.
In a natural way, then, base point free linear systems on
$\Mo{n}$ are identified
with linear systems on $\P^{n-3}$ whose base locus is
quite special and supported on
so-called vital spaces, i.e. spans of subsets of the given points.
Another feature of this picture is that all these
vital spaces correspond to divisors in $\Mo{n}$ which have
a modular interpretation as products of $\overline{M}_{0,r}$ for $r<n$.
In this interpretation, the modular forgetful maps
$\phi_I: \Mo{n} \to \overline{M}_{0,n-|I|}$ correspond, up to
standard Cremona transformations, to linear projections from vital spaces.

While, in \cite{BM}, the emphasis is put on the base locus of linear systems
in  $\P^{n-3}$ associated to fiber type morphisms on $\Mo{n}$, here we focus
on the geometry of fibres of fiber type morphisms.
In particular, the main Theorem of this paper concerns linear fibrations, which are dominant
morphisms $f: \Mo{n} \to X$ with connected fibres, whose general fiber is the blow-up
of a linear subspace in $\P^{n-3}$ via Kapranov's construction.

\begin{theoremn}
Let $f:\overline{M}_{0,n}\to X$ be a linear fibration.
Then there exist sets $I_1, \ldots, I_k \subset \{1, \ldots, n \}$ and a birational morphism
$g: \prod_{j=1}^k \overline{M}_{g,n-|I_j|} \to X$ 
such that $f= g \circ (\prod_{j=1}^k \phi_{I_j})$ factors through
the forgetful morphism $ \psi=\prod_{j=1}^k \phi_{I_j}$. 
\end{theoremn}

Unfortunately, linear fibrations do not exhaust the whole of fiber type morphisms.
Combining results of Sections 3 and 5 we recover the result of Farkas-Gibney and extend the classification
of fibrations on $\overline{M}_{0,n}$ for low $n$ to the case $n=7$. We prove:

\begin{theoremn}
Let  $f: \Mo{n} \to X$
be a dominant morphism with ${\rm dim}X < n-3$.
If $n \leq 7$
then  $f=g \circ h$, where $g$ is a generically finite
morphism and $h$ is a forgetful
morphism. More precisely $h$ is
one of the following  morphisms:
\begin{itemize}
\item[(i)] $\phi_I : \Mo{n} \to \overline{M}_{0,n-|I|}$, with $|I| \leq n-4$;
\item[(ii)] $\phi_I \times \phi_J :\overline{M}_{0,n} \to
\overline{M}_{0,n-|I|} \times \overline{M}_{0,n-|J|}$, where $|I|+|J| \leq 3$ and, if $|I \cap J| \geq 1$, then $n=7$ and $|I \cap J| =1$;
\item[(iii)] $ \phi_{\{i_1,i_2,i_3\}} \times \phi_{\{i_4,i_5,i_6\}} \times \phi_{\{i_1,i_4,i_7\}}:
\overline{M}_{0,7} \to \overline{M}_{0,4} \times \overline{M}_{0,4} \times
\overline{M}_{0,4}$.
\end{itemize}
\end{theoremn}

It is not clear to us how to deduce such statement from the very recent description of the 
Mori cone of $\overline{M}_{0,7}$ due to \cite {La}.\noindent

Case (iii) in above Theorem corresponds to a fibration on $\overline{M}_{0,7}$ whose fibers
are rational curves which cannot be reduced to lines by any standard Cremona transformation.
In Section 4 we give the example of a forgetful morphism on $\overline{M}_{0,8}$ whose fibers
are elliptic curves. We extend this example providing a series of forgetful morphisms on $\overline{M}_{0,2k+2}$
whose fibers are curves of arbitrary large genus $(k-4)2^{k-3}+1$. 
In order to understand these nonlinear examples, we are naturally
led to extend a criterion of Castravet-Tevelev in \cite{CT} in order 
to recognise surjectivity of forgetful morphisms.

\begin{theoremn}
Let sets $I_1, \ldots, I_k \subset \{1, \ldots, n \}$ be given, such that
if $i \ne j$, $I_i$ is not a subset of $I_j$.
The forgetful morphism
$\psi= \prod_{j=1}^k \phi_{I_j}: \Mo{n} \to \prod_{j=1}^k \overline{M}_{g,n-|I_j|}$
is surjective if and
only if for any $S \subseteq \{1, \ldots k \}$ we have:
$$ (*) \qquad n-|\cap_{j \in S}I_j|-3 \geq \sum_{j \in S}(n-|I_j|-3)$$
and $n-3-\sum_{j=1}^k(n-|I_j|-3)=:h \geq 0$.
\end{theoremn}

Finally, in Section 5, we give further evidence to the Fibration-Conjecture
in this way:

\begin{theoremn}
Let  $f: \Mo{n} \to X$
be a dominant morphism with ${\rm dim}X \leq 2$.
Then  $f=g \circ h$, where $g$ is a generically finite
morphism and $h$ is a forgetful
morphism. 
\end{theoremn}

The second named author would like to
thank Gavril Farkas for rising his
attention on Kapranov's paper,
\cite{Ka}, and the possibility to use
projective techniques in the study of $\Mo{n}$.
It is a pleasure to thank Gavril Farkas,
Angela Gibney,
 and Daniel Krashen, for their
enthusiasm  in a preliminary version of
these papers and for their suggestions to
clarify the contents. 

\section{Preliminaries}
\label{sec:notation}
We work over the field of complex numbers. An n-pointed curve
of arithmetic genus $0$ is the datum $(C; q_1, \ldots, q_n)$ of
a tree of smooth rational curves and $n$ ordered points on the
nonsingular locus of $C$ such that each component of $C$
has at least three points which are either marked or singular
points of $C$. If $n \geq 3$, $\Mo{n}$ is the smooth
$(n-3)$-dimensional scheme
constructed by Deligne-Mumford, which is the fine moduli scheme
of isomorphism classes $[(C; q_1, \ldots, q_n)]$ of stable n-pointed curves
of arithmetic genus $0$.

\noindent For any $ i \in \{1, \ldots, n \}$ the forgetful map
$$\phi_i : \Mo{n} \to \overline{M}_{0,n-1}$$
is the surjective morphism which associates
to the isomorphism class $[(C; q_1, \ldots, q_n)]$ of
a stable n-pointed rational curve $(C; q_1, \ldots, q_n)$
the isomorphism class of the $(n-1)$-pointed stable
rational curve obtained by forgetting $q_i$ and, if any,
contracting to a point a component of $C$ containing
only $q_i$, one node of $C$ and another marked point, say $q_j$.
The locus of such curves forms a divisor,
that we will denote by $E_{i,j}$. The morphism $\phi_i$ also plays the
role of the universal curve morphism, so that its
fibers are all rational curves transverse to
$n-1$ divisors $ E_{i,j}$.
Divisors $E_{i,j}$ are the images
of $n-1$ sections $s_{i,j}: \Mo{n-1} \to \Mo{n}$
of $\phi_i$. The section $s_{i,j}$ associates to
$[(C; q_1, \ldots,q_i^{\vee}, \ldots,  q_n)]$
the isomorphism class of the n-pointed stable
rational curve obtained by adding at $q_j$ a smooth rational curve
with marking of two points, labelled by $q_i$ and $q_j$.
Analogously, for every $I \subset \{1, \ldots, n \}$,
we have well defined forgetful
maps $\phi_I : \Mo{n} \to \overline{M}_{0,n-|I|}$.
>From our point of view the important part of a forgetful map is the set 
of forgotten index, more than the actual marking of the remaining. 
For this we slightly abuse the language and introduce the following definition.
\begin{definition}
\label{def:forg} A forgetful map is the composition of $\phi_I:\Mo{n}\to \Mo{r}$
 with an automorphism $g\in Aut(\Mo{r})$.
More generally a forgetful morphism is any product $\psi= \prod_{j=1}^k \phi_{I_j}$ of forgetful maps.
\end{definition}
\noindent In order to avoid trivial cases we will always
tacitly consider $\phi_I$ only if $n-|I| \geq 4$. Forgetful morphisms are basic examples
of morphisms on $\Mo{n}$. In this paper we will be interested mainly in forgetful morphisms
which are fiber type morphisms: 

\begin{definition} A dominant morphism $f:X\to Y$ is called a fiber type morphism if the dimension of the general fiber
is positive, i.e. $\dim X>\dim Y$. A fiber type morphism is said a fibration
if it has connected fibers. If $f$ is a fibration we denote by $F_f$ a general fiber of $f$.
\label{def:fibr}
\end{definition}

Besides the canonical class $K_{\Mo{n}}$, on
$\Mo{n}$ are defined line bundles $\Psi_i$ for each $i \in \{1, \ldots, n \}$ as follows:
the fiber of $\Psi_i$ at a point $[(C; q_1, \ldots, q_n)]$ is the tangent line $T_{C,p_i}$.
Kapranov, in \cite{Ka} proves the following:

\begin{theorem}
Let $p_1,\ldots,p_n \in \P^{n-2}$ be points in linear
general position. Let $H_n$ be the Hilbert scheme
of rational curves of degree $n-2$ in $\P^{n-2}$.
$\Mo{n}$ is isomorphic
to the subscheme $H \subset H_n$ parametrizing
curves containing $p_1, \ldots, p_n$.
For each $i \in  \{1,\ldots,n\}$ the line bundle
$\Psi_i$ is big and globally generated
and it induces a morphism $f_i : \Mo{n} \to \P^{n-3}$
which is an iterated blow-up
of the projections from $p_i$ of the given points
and of all strict transforms of the
linear spaces they generate, in order of increasing dimension.
\label{kapranov}
\end{theorem}

We will use the following notations:

\begin{definition} A Kapranov set
  $\K\subset\P^{n-3}$ is an ordered  set of $(n-1)$ points
  in linear general position, labelled
  by a subset of $\{1,\ldots, n\}$.
  For any $J \subset \K$, the linear span
  of points in $J$ is said a vital linear subspace
  of $\P^{n-3}$. A vital cycle is any union of vital
  linear subspaces.
\end{definition}

To any Kapranov set, labelled by
$\{1,\ldots, i-1, i+1,\ldots, n\}$,
is uniquely associated a Kapranov map,
$f_i:\Mo{n}\to\P^{n-3}$, with $\Psi_i= f_i^{*} \O_{\P^{n-3}}(1)$,
and to a Kapranov map is uniquely associated
a Kapranov set up to projectivity.

\begin{definition}\label{kap} Given a subset
  $ I= \{i, i_1,\ldots,i_s\} \subset \{1,\ldots,n\}$
  and the Kapranov map
  $f_i:\Mo{n}\to\P^{n-3}$, let
  $I^*=\{1,\ldots,n\}\setminus I$. Then we indicate
  with
$$H^
{i\vee}_{I^*}:=:V^i_{I\setminus\{i\}}:=:V^i_{i_1,\ldots,i_s}:=\langle p_{i_1},\ldots,
  p_{i_s}\rangle \subset \P^{n-3}$$ the vital linear subspace
  generated by the $p_{i_j}$'s and with
$$E_I:=E_{i,i_1,\ldots,i_s}:=f_i^{-1}(V^i_I)$$ the divisor associated on
  $\Mo{n}$.
\end{definition}
Notice that  $H^{h \vee}_{ij}$ is the
hyperplane missing the points $p_i$
and $p_j$ and the set
$$\K'=\K\setminus\{p_i,p_j\}\cup(
H^{\vee}_{ij}\cap\langle p_i,p_j\rangle)$$
is a Kapranov set in $H^{h \vee}_{ij}$.

 In particular for any $i \in \{1,\ldots,n\}$ and Kapranov
set $\K=\{p_1, \ldots, p_i^{\vee}, \ldots, p_n\}$ divisors
$E_{i,j}=f_{i}^{-1}(p_j)$ are defined and such notation is compatible
with the one adopted for the sections $E_{i,j}$ of $\phi_i$.
More generally, for any $i \in I \subset \{1, \ldots, n \}$ the divisor $E_I$
has the following property: its general
point corresponds to the isomorphism class of a rational
curve with two components, one with
$|I|+1$ marked points, the other with $|I^{*}|+1$ marked points,
glued together at the points not marked by elements of
$\{1,\ldots,n\}$.
It follows from this picture that
$E_I = E_{I^*}$ and that $E_I$ is abstractly isomorphic
to $\overline{M}_{0,|I|+1} \times \overline{M}_{0,|I^{*}|+1}$. The divisors
$E_I$ parametrise singular rational curves, and
they are usually called boundary divisors.
A further property of $E_I$ is that for each choice
of $i \in I, j \in I^{*}$, $E_I$ is a section of the forgetful morphism:
$$ \phi_{I \setminus \{ i \} } \times \phi_{I^{*} \setminus \{ j \} } :
 \Mo{n} \to \overline{M}_{0,|I^{*}|+1} \times \overline{M}_{0,|I|+1}.$$
This morphism is surjective and all fibers are rational curves.
With our notations $f_i(E_I)$ is a vital linear space of dimension $|I|-2$
if $i \in I$ and a vital linear space of dimension $|I^{*}|-2$ if $i \notin I$.

\begin{definition}\label{def:cre} Let $\K$ be a
  Kapranov set with Kapranov map
  $f_i:\Mo{n}\to\P^{n-3}$. Then let
$$\omega^\K_j:\P^{n-3}\rat\P^{n-3}$$
 be the standard Cremona transformation centered on
  $\K\setminus \{p_j\}$. Via Kapranov's
construction we can associate a Kapranov
set  labelled by
$\{1,\ldots,n\}\setminus\{j\}$ to the
rhs $\P^{n-3}$ and in this notation
Kapranov, \cite[Proposition 2.14]{Ka} proved that
$$\omega_j^\K= f_j\circ f_i^{-1}$$ as
birational maps.
By a slight abuse of notation we can define
$\omega_j^\K(V^i_I):=f_j(E_{I,i})$,
even if $\omega_j$ is not defined on the
general point of $V^i_I$.
\end{definition}

\begin{remark}\label{rem:kcre}
Let $\omega^\K_h$ be the standard
Cremona transformation centered on
$\K\setminus\{p_h\}$, and
$\K'$ the Kapranov set associated to the hyperplane $H^{i
    \vee}_{jk}$. Then for $h\neq
j,k$ we have  $\omega^\K_{h|H^{i
    \vee}_{jk}}=\omega^{\K'}_h$. This
extends to arbitrary vital linear spaces.
It follows from the definitions that
$\omega_h^\K(V^i_I)=V^h_{I \setminus \{h\},i }$
if $h \in I$ and $\omega_h^\K(V^i_I)=V^h_{(I\cup\{i\})^{*}
\setminus \{h\}}$ if $h \in I^{*}$.
\end{remark}

We are interested in describing linear systems on $\P^{n-3}$ that are
associated to fiber type morphisms on
$\Mo{n}$. For this purpose we introduce
some definitions.
\begin{definition}\label{def:s} A
  linear system  on a smooth projective variety
  $X$ is uniquely determined
  by a pair $(L,V)$, where $L\in Pic(X)$
  is a line bundle and $V\subseteq
  H^0(X,L)$ is a vector space. If no
  confusion is likely to arise we will
  forget about $V$ and let
  $\L=(L,V)$.  Let
  $g:Y\to X$ be a birational morphism
  between smooth varieties. Let $A\in\L$ be a general
  element and $A_Y=g^{-1}_*A$ the strict
  transform. Then  $g^*L=A_Y+\Delta$ for
  some  effective  $g$-exceptional
  divisor $\Delta$.  The strict transform of
  $\L=(L,V)$ via $g$ is
$$g^{-1}_*\L:=(g^*L-\Delta,V_Y)$$
where $V_Y$ is the vector space spanned by the
strict transform of elements in $V$.
\end{definition}

\begin{definition}\label{def:Mk}  Let $\K\subset\P^{n-3}$
  be a Kapranov set and
  $f_i:\overline{M}_{0,n}\to\P^{n-3}$
  the associated map. An $M_{\K}$-linear system
  on $\P^{n-3}$    is a linear system
  $\L\subseteq|\O_{\P^{n-3}}(d)|$, for some $d$, such that
  $f_{i*}^{-1}\L$ is a base
  point free linear system.
 \end{definition}

A further property inherited from Kapranov's
construction is the following.

\begin{remark} \label{rem:ka}
Let $H^{h \vee}_{ij}$ be the
hyperplane missing the points $p_i$
and $p_j$ and $$\K'=\K\setminus\{p_i,p_j\}\cup
(H^{\vee}_{ij}\cap\langle p_i,p_j\rangle)$$
the associated Kapranov set.
If $\L$ is an $M_\K$-linear system then $\L_{|H^{\vee}_{ij}}$
is an $M_{\K'}$ linear system.
\end{remark}

Basic examples of fiber type morphisms are forgetful maps.
Consider any set $ I \subset \{1, \ldots ,n \} $ and the associated
forgetful map $\phi_I$. If $j \notin I$ a typical diagram we will consider
is
\[
\xymatrix{
\overline{M}_{0,n}\ar[d]^{f_{j}}\ar[r]^{\phi_I}
&\overline{M}_{0,n-|I|}\ar[d]^{f_{h}} \\
\P^{n-3}\ar@{.>}[r]^{\pi_{I}}&\P^{n-|I|}}
\]
where the $f_j$ and $f_h$ are Kapranov maps and
$\pi_{I}$ is the projection from $V^j_{I }$.
In this case an $M_\K$-linear system
associated to $\Phi_I$ is given by
 $|\O(1)\otimes\I_{V^j_{I }}|$
and 
$f_j(F_{\phi_I})$ is a linear space of
dimension $|I|$.

\begin{lemma} \label{lem:fibers}
Let $ I \subset \{1, \ldots ,n \} $ and let  $\phi_I$ be the associated
forgetful map. If $j \notin I$ then $f_j(F_{\phi_I})$ is a linear space of
codimension $|I|$. If $j \in I$ then $f_j(F_{\phi_I})$ is the cone with
vertex $V^j_{I \setminus \{j \}}$ over a rational normal curve of degree
$n-|I|-2$ through $ \K \setminus (I \cap \K)$. If $\psi= \prod_{j=1}^k \phi_{I_j}$
is any surjective forgetful morphism $F_{\psi}= \cap_{j=1}^k F_{\phi_{I_j}}$.
\end{lemma}
\begin{proof}
Since, for any $i \in \{1, \ldots, n \}$, and any $i \in I$, we have 
a factorisation $\phi_I=  \phi_{I \setminus \{ i\}} \circ \phi_i$, it is enough
to understand the geometry of $f_i(F_{\phi_i})$. 
But $f_i(F_{\phi_i})= \omega_i^\K(f_j(F_{\phi_i}))$ and
for any $ j \ne i$ $f_j(F_{\phi_i})$ is a line through the point $p_i$;
it then follows that  $f_i(F_{\phi_i})$ is a rational normal curve 
of degree $n-3$ passing through $K$.
The rest follows easily from this.
\end{proof}

\noindent Kapranov's description of $\Mo{n}$
easily allows us to describe generators and relations
for the Picard group of line bundles on $\Mo{n}$; 
in next proposition we prove by projective techniques
a formula by Pandharipande, see \cite{Pa}:

\begin{proposition}
The Picard group of line bundles on $\Mo{n}$ is
generated by $K_M$ and the boundary divisors $E_I$.
The following relations hold:
\begin{itemize}
\item[(i)]for every $i  \in \{1, \ldots, n \}$,
$K_M=(2-n)\Psi_i + \sum_{i \in I}(n-2-|I|)E_{I};$
\item[(ii)] for every $i \ne j \in \{1, \ldots, n \}$,
$\Psi_i+\Psi_j=\sum_{i,j \in I}(n-2-|I|)E_I-K_M;$
\item[(iii)] for every $i  \in \{1, \ldots, n \}$,
$ \sum_{i \in I} [(|I|-2)(n-2-|I|) -2]E_I -(n-1)K_M=0.$
\end{itemize}
\label{pic}
\end{proposition}
\begin{proof}
Point (i) is clear in view of the
above description of $\Mo{n}$ as a blow up
of $\P^{n-3}$ at all vital linear subspaces
spanned by a Kapranov set $K$.

To get (ii) note that by definition of
the Cremona transformation $\omega_j^\K$
we have $\omega_j^\K= f_j\circ f_i^{-1}$
and 
\begin{equation}\Psi_j= (n-3) \Psi_i - \sum_{i  \in I,
  j \notin I}(n-2-|I|)E_{I}.
\label{eq:213}
\end{equation}
These plugged in formula (i) give the
required equation.

Formula (iii) is more complicate.
First note that for every $i \ne j
\in \{1, \ldots, n \}$, we have:

\begin{equation} \qquad \Psi_i-\Psi_j =
  \sum_{i \in I, j \notin I
  }(n-2|I|)E_{I}.
\label{eq:1}
\end{equation}

To show this choose $h \ne i, j$ and
compute $\Psi_i$ and $\Psi_j$ via the pullback
of the hyperplanes $H^{i\vee}_{j,h}$ and
$H^{j\vee}_{i,h}$. This yields

$$\Psi_i= \sum_{i \in I, h,j \notin I}E_I,
\qquad \Psi_j= \sum_{j \in I, h,i \notin I}E_I. $$
The required equality is then obtained
summing over all $h\ne i,j$, and 
  passing to the complementary
indexes whenever either $i$ is not contained in $I$.

Next sum equation
(\ref{eq:1})  over all $j$ to get:
\begin{equation} \qquad (n-2) \sum_{j \neq i}(\Psi_i-\Psi_j) =
\sum_{i \in I}(n-|I|)(n-2|I|)E_I.\label{eq:2}
\end{equation}

Then equation (\ref{eq:213}) yields
$$\Psi_i -\Psi_j=-(n-4)\Psi_i+\sum_{i \in I, j \notin I}(n-2-|I|)E_{I}.$$
This summed over all $j$ gives:
$$ \sum_j(\Psi_i-\Psi_j)=-(n-4)(n-1)\Psi_i + \sum_{i \in I} \ (n-|I|)(n-|I|-2)E_I.$$
Multiply the latter by $(n-2)$ and  equate
 to equation (\ref{eq:2}). This together
 with the following identity 
$$(n-2)(n-|I|)(n-|I|-2)-(n-|I|)(n-2|I|)=(n-|I|)(n-4)(n-|I|-1)$$
gives:

\begin{equation} \qquad (n-1)(n-2)\psi_i= \sum_{i \in I} (n-|I|)(n-|I|-1)E_I.
\label{eq:3}
\end{equation}

To conclude multiply  formula in  (i) by
$(n-1)$ and equate with equation
(\ref{eq:3}) to get
the required
 $$(n-1)(n-|I|-2) -(n-|I|)(n-|I|-1)= (|I|-2)(n-|I|-2)-2.$$
\end{proof}


\section{Linear fibrations}
Recall from Definition \ref{def:fibr} that if $f: \Mo{n} \to X$ is
a fibration we denote by
 $F_f$ a general fiber of $f$.

\begin{definition}
A fibration $f$ is said a linear fibration if there exists
$i \in \{1, \ldots, n \}$ such that $f_i(F_f)$ is a linear space.
\label{def:linf}
\end{definition}
If $\pi: \P^{n-3} \to X$ is a rational fibration with general
fiber a linear space, and if there exists a Kapranov set $K \subset \P^{n-3}$ with
Kapranov map $f_i: \Mo{n} \to \P^{ n-3 }$ such that the composition $f= \pi \circ f_i$
is a morphism, then $f$ is a linear fibration. \par \noindent

In this section we will focus on linear fibrations and we will show that every linear fibration
on $\Mo{n}$ factors through a forgetful morphism. \par \noindent

Suppose that a forgetful morphism $$\psi= \phi_{I_1} \times \cdots \times \phi_{I_k}: \Mo{n} \to
\overline{M}_{0,n-|I_1|} \times \cdots \times \overline{M}_{0,n-|I_k|}$$ is a fibration and that
there exists $i \in {1, \ldots, n}$ such that $i \notin \cup_{j=1,\ldots,k} I_j$.
Then $\psi$ is a linear fibration.
In fact we have a commutative diagram:
\[
\xymatrix{
\overline{M}_{0,n}\ar[d]^{\psi}\ar[r]^{f_i}&
\P^{n-3} \ar@{.>}[d]^{\pi_{V^i_{I_1}}\times \cdots \times \pi_{{V^i_{I_k}}}}\\
\overline{M}_{0,n-|I_1|} \times \cdots \times \overline{M}_{0,n-|I_k|}\ar[r]^{f_i}&\P^{n-3-|I_1|} \times \cdots \times \P^{n-3-|I_k|}}
\]
and  by Proposition \ref{lem:fibers}
$f_i(F_{\psi})$ is a linear space.
In this situation, Castravet and Televev, in \cite[Theorem 3.1]{CT},
 give a necessary and sufficient condition
for such a $\psi$ to be surjective, which is what is needed in order
for $f$ to be a fibration. A first
obvious condition for the surjectivity
of $\psi$ is that there are no inclusions among any two sets $I_i$ and $I_j$
of the collection $I_1, \ldots, I_k$.  This assumption made, they prove:

\begin{proposition}
With the above hypotheses, the morphism $\psi$ is surjective if and only if
for any $ S \subset \{1, \ldots, k \}$,
$$(*) \qquad n-|\cap_{j \in S}I_j|-3 \geq \sum_{j \in S}(n-|I_j|-3)$$
and $n-3-\sum_{j=1}^k(n-|I_j|-3)=:h \geq 0$.
\label{ct}
\end{proposition}

Sometimes it turns out to be useful to
factorize forgetful maps with simpler 
ones. This is our next task. To make
clearer what we mean let us start with a
simple example.
Consider the forgetful morphism
$$ \psi=\phi_{\{i_1,i_2\}} \times \phi_{\{i_2,i_3\}}: \overline{M}_{0,6} \to \overline{M}_{0,4} \times \overline{M}_{0,4},$$
with $1 \notin \{i_1,i_2,i_3 \}$: $f_i(F_{\psi})$
intersects lines $\langle p_{i_1}, p_{i_2} \rangle$
and $\langle p_{i_2}, p_{i_3} \rangle$ only at $p_{i_2}$,
because otherwise it would lie on the plane $\langle p_{i_1}, p_{i_2}, p_{i_3} \rangle$; notice that this plane
is contracted to a point by $\psi$. An equivalent way of seeing this fact is that
we have a factorisation $\psi=(\phi_{i_1} \times \phi_{i_3}) \circ \phi_{i_2}$, where $\phi_{i_1} \times \phi_{i_3}$
is birational.
We generalise this fact in the following Proposition.

\begin{proposition}
Suppose that $I_1, \ldots, I_k$ are subsets of $\{1,\ldots, n\}$, with
no inclusions among them, no one of them containing an index, say $1$, 
and satisfying condition $(*)$. Let $\psi= \prod_{j=1}^k \phi_{I_j}$; then
$f_1(F_f)$ intersects the vital spaces $V^1_{I_j}$ at their general point
if and only if for any $S \subseteq \{1, \ldots k \}$ we have:
$$ (**) \qquad n-|\cap_{j \in S}I_j|-3 > \sum_{j \in S}(n-|I_j|-3).$$
If condition $(**)$ is not satisfied, there exists a collection
$J_1, \ldots, J_l$ of subsets of $\{1,\ldots,n\}$, with
no inclusions among them, no one of them containing $1$, 
and satisfying condition $(**)$, such that for every $h=1,\ldots,l$
there exists $m_h \in \{1,\ldots, k\}$ such that $J_h \subset I_{m_h}$ and, if
$h=\prod_{j=1}^l \phi_{J_j}$, a factorisation $\psi=g \circ h$ is given,
with $g$ a birational morphism.
\label{**}
\end{proposition}
\begin{proof}
If $I_1, \ldots, I_k$ are subsets of $\{1,\ldots, n\}$, with
no inclusions among them, no one of them containing an index, say $1$, 
and satisfying condition $(*)$, then  $\psi= \prod_{j=1}^k \phi_{I_j}$
is surjective and so is any partial projection $\psi_S:=\prod_{j \in S} \phi_{I_j}$.

Now, condition $(**)$ is not satisfied if and only if
there exists $S \subset \{ 1, \ldots, k \}$ such that
$$ n-|\cap_{j \in S}I_j|-3 = \sum_{j \in S}(n-|I_j|-3).$$

Since $|\cap_{j \in S}I_j|= n-3-\sum_{j \in S}(n-|I_j|-3)$ is the dimension of 
the general fiber of $\psi_S$, this condition is equivalent to a factorisation of $\psi_S$
through $\phi_{\cap_{j \in S}I_j}$ and this is equivalent
to the condition that $f_i(F_{\psi})$ does not intersect 
some vital space  $V^i_{I_j}$ at its general point.
\end{proof}

It is easy to show that condition $(**)$ on sets $I_1, \ldots, I_k$ as above
is equivalent to the condition that for every $S \subset \{1, \ldots, k \}$
the linear span $\langle V^1_{I_j} |j
\in S \rangle$ is as large as
possible. \par \noindent

 The linearity assumption, although
very restrictive, 
allows us to prove the best possible
result towards factorisation of fibrations on $\Mo{n}$. 

\begin{theorem}
Let
  $f:\overline{M}_{0,n}\to X$ be a linear fibration.
  Then there exist sets $I_1, \ldots, I_k \subset \{1, \ldots, n \}$ and a birational morphism
  $g: \prod_{j=1}^k \overline{M}_{g,n-|I_j|} \to X$ 
  such that $f= g \circ (\prod_{j=1}^k \phi_{I_j})$ factors through
  the forgetful morphism $\prod_{j=1}^k \phi_{I_j}$. 
\label{th:linear}
\end{theorem}
\begin{proof}
Fix a very ample linear system $\A\in Pic(X)$.
 Let $\L=f^*(\A)$ and $\L_1=f_{1*}(\L)$, an $M_K$ linear system on $\P^{n-3}$.
By hypothesis we may assume that
$f_1(F_f)=:L$ is a linear space of 
dimension $k$.  
Since $F_f$ is the complete intersection of elements of $\L$,
and since $F_f$ is a certain blow-up of $L$,
if we restrict the canonical bundle to $F_f$
we obtain by adjunction:
$$K_{\Mo{n}}|_{F_f}=-(k+1)\Psi_1|_{F_f}+\Delta,$$
for some effective and exceptional divisor $\Delta\in Pic(F_f)$.
More precisely, $\Delta= \sum_{1 \in I}(k+1-|I|)E_I|_{F_f}.$
On the other hand, by Proposition
\ref{pic} (i), restricted to $F_f$, we get
$$ (++) \qquad (n-3-k)\Psi_1|_{F_f}+\Delta= \sum_{1 \in I}(n-2-|I|)E_I|_{F_f}.$$

Since $\Psi_1$ and exceptional classes are independent in $Pic(F_f)$,
$(++)$ tells us that there are sets $I_1, \ldots, I_l$ not containing 
the index $1$ such that 
$$ E_{I_j \cup \{1\} }|_{F_f} = \Psi_1|_{F_f},$$
so that $V_{I_j}^1$ intersects $L$ in codimension one.
Moreover, for those sets, $(++)$ implies that

$$(n-3-k)= \sum_{j=1}^l(n-3-|I_j|).$$
This shows that the linear space
$L$ is the intersection of fibers
of projections with centers $V^1_{I_j}$,
for $j=1,\ldots,l$; in other words, $L$
is also a fiber of the rational map
$\pi_{V^1_{I_1}}\times \cdots \times \pi_{V^1_{I_l}}$.
Notice that if we knew that 
$\psi= \phi_{I_1}\times \cdots \times \phi_{I_l}$
was equidimensional, we could deduce
already from this that $f$ factors
through $\psi$. 

 Since $L$ does
not lie in a vital hyperplane,
sets $I_1, \ldots, I_l$ do satisfy
condition $(**)$. In particular,
for any pair of different indexes $h,m$, we
have $\langle
V^1_{I_h},V^1_{I_m}\rangle=\P^{n-3}$. Let us
stress the main consequence of this.
 Let
$\epsilon_1:Y_1\to \P^{n-3}$ be the blow
up of $V^1_{I_1}$ and $\M_j:=|\O(1)\otimes
I_{V_{I_j}^1}|$. Then we have
$$\epsilon_{1*}^{-1}
V^1_{I_j}=\Bl\epsilon_{1*}^{-1}\M_j.$$
Let $Y$ be the blow up of $\P^{n-3}$ along
the linear spaces
$\{V^1_{I_j}\}_{j\in\{1,\ldots l\}}$ in order 
of decreasing dimension, with
morphism $a:Y\to\P^{n-3}$. Let
$L_Y=a^{-1}_*(L)$ be the strict
transform of the general fiber $L$. Then
for any $j=1,\ldots, l$ we have
a morphisms $\f_j:Y\to \P^{n-3-|I_j|}$,
given by $a_*^{-1}\M_j$. Moreover
$L_Y$ is contained in a general fiber of each $\f_j$.
 This yields
a morphism onto the product
$$\f:Y\to W=\prod_{j=1,\ldots l}\P^{n-3-|I_j|}$$
with $\rk Pic(Y/W)=1$ and  general fiber
$L_Y$. The latter two, together with the fact that
$\f$-numerical and $\f$-linear
equivalence coincide, allow to conclude
that there exists $\H_1\in Pic(W)$ such that
$$a^{-1}_*(\L_1)=\f^*(\H_1).$$
If $\dim W=1$ then this is enough to
conclude.

 Assume that $\dim W>1$.
Let $\chi=\f\circ a^{-1}:\P^{n-3}\rat
W$ be the induced rational map.  Consider the
following commutative diagram induced by
the forgetful map $\psi$,

\[
\xymatrix{
\overline{M}_{0,n}\ar[d]^{f_1}\ar[rr]^{\hspace{-.5cm}\psi=\prod
\phi_{I_j}}
&&\prod_{j=1,\ldots,l} \Mo{n-|I_j|} \ar[d]^{f_{i_1} \times\cdots\times f_{i_l}} \\
\P^{n-3}\ar@{.>}[rr]^{\chi}&&
W.}
\]

 The linear system
$\L$ is base point free and numerically
trivial on $F_{\psi}$, the general fiber of
$\psi$. Therefore $\L$ is $\psi$-numerically
trivial. Moreover $\psi$ is a forgetful
morphism with rational fibers and therefore $\psi$-numerical and
$\psi$-linear equivalence coincide. This
shows that
$$\L=\psi^*(\H),$$ for some base point free
linear system $\H\in
Pic(\prod_{j=1,\ldots,l} \Mo{n-3-|I_j|})$,
and yields
that there is a birational morphism 
$$g:\prod_{j=1,\ldots,l} \Mo{n-3-|I_j|}\to X$$ 
induced by the linear system $\H$ such
that
$f=g\circ \psi$.
\end{proof}

\begin{remark} As soon as $n\geq 7$
  there are forgetful maps that are not
  linear and even with non
  rational fibers, see example
  \ref{ex:elliptic}.
\end{remark}

\section{Fibrations with low dimensional source}
\label{sec:num}
 
In this section we will be mainly interested in fibrations by curves.
\begin{definition}
A fibration $f: \overline {M}_{0,n} \to X$ is a fibration by curves if
$F_f$ is a connected curve.
\end{definition}
\begin{definition}\label{numbers}
Let $f$ be a fibration by curves and $\L$ the associated basepoint free linear system
on $\overline {M}_{0,n}$. Then $\L^{n-3}=0$ and for every
$I \subset\{1,\ldots,n\}$ we denote
$$m_I:=\L^{n-4} \cdot E_I\in \N$$
For every $i \in \{1, \ldots, n \}$, we  denote
$$C_i:=f_i(F_f).$$
Also, if $i \in \{1, \ldots, n \}$ and if $f_i: \overline {M}_{0,n} \to \P^{n-3}$
is a Kapranov map, we denote $$d_i:= \Psi_i \cdot \L^{n-4}.$$
\end{definition}
These numbers have the following interpretation: 
$C_i \subset \P^{n-3}$ is a a curve of degree $d_i$,
which intersects with some non negative multiplicity $m_I$ 
the general point of any vital space $f_i(E_I) \subset \P^{n-3}$.
An important point in this definition is the following: while
$m_I$ is defined on $\overline {M}_{0,n}$, it really says something
about the intersection of any $C_i$ with vital spaces in $\P^{n-3}$.
We summarize these observations in the following.

\begin{lemma}\label{switch}
Suppose that $f$ is a fibration by curves. 
Then, for any $i \in I$, $C_i$ intersects with multiplicity $m_I$ any vital
space $V^i_{I\setminus\{i\}}$. If $i \notin I$, $m_I$ is the multiplicity
of $C_i$ at the general point of $V^i_{I^*\setminus\{i\}}$.
Furthermore, if $i \in I, j \notin I$, $C_i$ intersects with muliplicity $m_I$ the general
point of $V^i_{I\setminus\{i\}}$ and $C_j=\omega_j^\K(C_i)$ intersects
with multiplicity $m_I$ the general point of $V^j_{I^*\setminus\{j\}}=\omega_j^\K(V^i_{I\setminus\{i\}})$
\end{lemma}
\begin{proof}
It comes directly from Definitions \ref{kap} and \ref{numbers}, and Remark\ref{rem:kcre}. 
\end{proof}
 
A corollary of Proposition \ref{pic} is the following:

\begin{lemma}\label{lem:conti}
Let $f$ be a fibration by curves with connected fibers.
Let $g$ be the arithmetic genus of $F_f$; the following identities hold:
\begin{itemize}
\item[(i)]for every $i  \in \{1, \ldots, n \}$, $(n-2)d_i= \sum_{i \in I}(n-2-|I|)m_I+2-2g;$
\item[(ii)] for every $i \ne j \in \{1, \ldots, n \}$, $d_i+d_j=\sum_{i,j \in I}(n-2-|I|)m_I+2-2g;$
\item[(iii)] for every $i  \in \{1, \ldots, n \}$,
$ \sum_{i \in I} [(|I|-2)(n-2-|I|) -2]m_I +(n-1)(2-2g)=0.$
\end{itemize}
\end{lemma}
\begin{proof}
The general fiber $F_f$ is a complete intersection in a smooth variety,
because it is $\L^{n-4}$ and $L$ is basepoint free;
hence it has a dualizing sheaf $\Cal \omega_{F_f}$, whose
degree is $2g-2$; by adjunction 
$$(K_M+(n-4)\L) \cdot \L^{n-4} = \Cal \omega_{F_f}.$$
Then it is enough to restrict the identities in Proposition \ref{pic} to $F_f$.
\end{proof}

\begin{lemma}\label{lem:restrizione}
Let $f: \Mo{n} \to X$ be a fibration by curves and let $I \subset \{1, \ldots, n \}$
be such that $|I|=n-3$.
 Suppose that $E_I$ intersects $F_f$. Then for every
$J \supset I$, with  $|J|=n-2$,  $f(E_J)$ has codimension at most $1$ in $X$.
In particular, if $m_I \ne 0$ and $m_J=0$, $f$ restricts to a fiber type morphism
of general relative dimension one on $E_J$.
\end{lemma}
\begin{proof}
Let us consider the divisor $E_I\iso
\Mo{4}\times\Mo{n-2}$. Let $E_x$ be a
fiber of the canonical projection onto
the first factor. In this notation we have $E_I\cap
E_J=E_x$. By hypothesis we have
$f(E_I)=X$ hence $f(E_x)$ is a divisor
and consequently $\cod f(E_J)\leq 1$.

Assume that $m_I \ne 0$ and $m_J =0$.
Then $E_I$ intersects the general fiber
of $f$ while  $f(E_J)$ cannot dominate
$X$. Therefore $\cod f(E_J)=1$ and  $f$ restricts
to a morphism of fiber type of relative dimension one on $E_J$.
\end{proof}

Notice that in Lemma \ref{lem:restrizione} the restricted
morphism $f_{|E_J}$ fails to be a fibration by curves iff
it has non connected fibers.

We will use the following:
\begin{lemma}\label{lem:zero}
A fibration by curves $f$ factors through the forgetful map $\phi_i$
for some $i \in \{1, \ldots,n \}$ if and only if 
 $n=5$ or $n \ge 6$ and $m_I=0$ for all $I$ such that
$3 \leq |I| \leq n-3$.
\end{lemma}
\begin{proof}
If $f$ factors through $\phi_i$, for every 
$j \ne i$, $f_j(F_f)$ is a line through
the point $p_i$ and then $m_I=0$ for all $I$ such that
$3 \leq |I| \leq n-3$.

\noindent Suppose conversely that $n =5$ or that 
$n \ge 6$ and $m_I=0$ for all $I$ such that
$3 \leq |I| \leq n-3$. 

From Lemma \ref{lem:conti}(iii),  
$ \sum_{|I|=2}m_I =(n-1)(1-g),$ hence $g=0$ and
for each $j \in \{1, \ldots,n \}$, from Lemma \ref{lem:conti} (i): 
$$(n-2)d_j= \sum_{j \in I,
  |I|=2}(n-4)m_I+2.$$ On the other hand,
$\sum_{j \in I, |I|=2}m_I \le
\sum_{|I|=2}m_I =(n-1)$, so that we have
an upper bound $$d_j \leq n-3$$ for all
$j \in \{1, \ldots,n \}$.  Since $f$ has
connected fibers and $f_j(F_f)$ is a
rational curve, this implies that for
each $I \subset \{1, \ldots,n \}$, $m_I
\leq 1$.  Hence we may assume that
there exists $i \in \{1, \ldots,n \}$
such that $m_{ \{i,j \}}=1= \sum_{j \in
  I, |I|=2}m_I$, so that $f_j(F_f)$ is a
line through $p_i$. Then $f$ factors
through $\phi_i$ by Theorem \ref{th:linear}
\end{proof}

Suppose next that $f$ is a fibration by curves such that
the restriction to some divisor $E_I$ is still a morphism
of fiber type with relative dimension one.
We need a way to relate numerical characters of $f$ and numerical characters
of $f|_{E_I}$.
\begin{lemma}\label{lem:restnum}
Suppose that $f$ is a fibration by curves and suppose that there exists $I \subset \{1, \ldots, n \}$
with $|I|=n-2$ such that $g:=f|_{E_I}$ is a fiber type morphism of relative dimension one. 
Let $i \in I$, so that in the identification $E_I =\Mo{n-1}$ the restriction $f_i|_{E_I}$ is a Kapranov map to 
the hyperplane $H^
{i\vee}_{I^*}=V^i_{I\setminus\{i\}} \subset \P^{n-3}$. Then $f_i(F_g)$ is the union of components of 
a fiber of $f$ and ${\rm deg}f_i(F_g) \leq \sum_{J \subset I}m_J$. Equality holds if and only if $F_g$ is a fiber of $f$. Moreover, if $m^{'}_J$ denotes
the intersection with $E_J \cap E_I$ of the general fiber of $F_g$,
we have $m^{'}_J \leq m_J$.
\end{lemma}
\begin{proof}
The hypothesis means that the restriction of $\L$ to $E_I$ is a fiber type
morphism whose general fiber has dimension one.
This means that $\L^{n-5} \cdot E_I=F_g$ is a curve in $E_I$. Via the Kapranov
map $f_i$ (see definition \ref{def:cre}) $f_i(F_g)$ is clearly contained in a fiber of $f$.
Since $f(E_I)$ is a divisor in $X$, the general fiber $F_g$ is contained in a curve
so that it is in fact union of  components of such a curve. It then follows that,
since $f_i(E_I)=H^
{i\vee}_{I^*}= V^i_{I\setminus\{i\}}$ is a hyperplane in $\P^{n-3}$, 
$${\rm deg}f_i(F_g) \leq d_i=m_I +\sum_{J \subset I}m_J = \sum_{J \subset I}m_J,$$
because $m_I=0$. It is evident that equality holds if and only if ${\rm deg}f_i(F_g) = d_i$.
The final part of the statement comes from the fact that, as algebraic classes
in $\Mo{n}$, we have an inclusion $\L^{n-5} \cdot E_I \subset \L^{n-4}$,
so that for any $J$, we have
$$ m^{'}_J= \L^{n-5} \cdot E_I \cdot E_J \leq \L^{n-4} \cdot E_J.$$
\end{proof}
\begin{definition}\label{def:relev}
In the situation of Lemma \ref{lem:restnum}, we say that
a vital subspace of a hyperplane $f_i(E_I)$, is relevant if it
is the span of Kapranov points both for $\P^{n-3}$ and $f_i(E_I)$.
\end{definition}
Notice that formula ${\rm deg}f_i(F_g) \leq \sum_{J \subset I}m_J$
of Lemma \ref{lem:restnum} only contains intersection numbers
with relevant vital subspaces of $f_i(E_I)$. \par \noindent

We are able to classify fibrations by curves for low $n$.
The techniques we are going to use are the following: first of all,
if $n$ is small, formulae from Lemma \ref{lem:conti} will give us upper bounds
for the genus and the degree of projective Kapranov embeddings of fibers
of a fibration $f$. If needed we will look and will find hyperplanes
to which $f$ restricts to a fiber type morphism of relative dimension one.
Main techniques here will be Lemmata \ref{switch}, \ref{lem:restrizione}
and \ref{lem:restnum}.
\begin{proposition} \label{prop:curves}
Let $n \leq 7$, let $\{1, \ldots, n \}= \{i_1, \ldots, i_n \}$ and let
$f$ be a fibration by curves. Then there exists a factorisation
$f=g \circ h$ where $g$ is a birational morphism and $h$ is one of the following:
\begin{itemize}
\item[(i)] $h=\phi_{i_j}$ for some $j \in \{1,\ldots,n\}$,
\item[(ii)] $h= \phi_{\{i_1,i_2\}} \times \phi_{\{i_3,i_4\}}$ and $n=6$,
\item[(iii)] $h= \phi_{\{i_1,i_2\}} \times \phi_{\{i_3,i_4,i_5\}}$ and $n=7$,
\item[(iv)] $h= \phi_{\{i_1, i_2, i_3\}} \times \phi_{\{i_4,i_5,i_6\}} \times \phi_{\{i_1,i_4, i_7\}}$ and $n=7$.
\end{itemize}
Moreover if $n<7$, $f$ is a fibration by lines.
\end{proposition}
\begin{proof}
In view of Lemma \ref{lem:zero} we must prove the
result only for $n=6$ and $n=7$. For
$n=6$ the result is already known,
\cite{FG}, but we prefer to restate it
in our terminology, as a warm up.

\noindent In case $n=6$, (iii) of Lemma \ref{lem:conti} gives a sum where all the coefficients of $m_I$ are not
positive. This implies that $g=0$, i.e. that $F_f$ is a rational curve. 
>From Lemma \ref{lem:conti} (i) and (iii) we get:
$$4d_i= \sum_{i \in I, |I|=2}2m_I+ \sum_{i \in I, |I|=3}m_I+ 2,$$
$$ 2\sum_{|I|=2}m_I + \sum_{i \in I, |I|=3}m_I=10,$$
from which we derive that for all $i \in \{1, \ldots, 6 \}$ we have $d_i \leq 3$.
If $d_1=1$ we are done by Theorem \ref{th:linear}. So we will assume $d_1 \in \{2,3\}$

\noindent Let us fix a Kapranov set $\{p_2, \ldots, p_6 \}$ in $\P^3$ and let $C_1=f_1(F_f)$.
\noindent If $d_1=2$, $C$ is a plane conic so that, since $F_f$ is connected, for all $I$ with
$|I|=2$ and $1 \in I$, $m_I \leq 1$. From Lemma \ref{lem:zero}, if $f$ is not
a fibration by lines, up to renumbering indexes, only the following cases can occur:
\begin{itemize}
\item[(i)] $m_{ \{12 \}}=m_{ \{13 \}}=1$, $m_{ \{1i \}}=0$ for all other $i$,
\item[(ii)] $m_{ \{12 \}}=1$, $m_{ \{1i \}}=0$ for all other $i$,
\item[(iii)] $m_{ \{1i \}}=0$ for all  $i$.
\end{itemize}
In the first case, from Lemma \ref{lem:conti} (ii) follows that
for every $i \geq 4$ there exist at least two sets $I$ with $1,i \in I$ and $|I|=3$
for which $m_I \ne 0$ but we know that $\sum_{1 \in I, |I|=3}m_I=2$ and this is impossible.

\noindent In the second case, as above, for every $i \geq 3$ there exist at least two sets $I$ with
$1,i \in I$ and $|I|=3$ for which $m_I \ne 0$; since the general conic doesn't intersect a plane in more than three points, up to
rearranging indexes,we must have that only for
$I= \{1,3,4 \}, \{1,3,5 \}, \{1,5,6 \}, \{1,4,6 \}$ is $m_I=1$,
while $m_I=0$ for any other $I$ with $|I|=3$. The other non zero $m_I$ must then correspond to
$m_{ \{4,5 \}}=m_{ \{3,6 \}}=1$, so that $C$ passes through one point and intersects
four lines which lie on a quadric as a curve of type $(2,2)$; in particular, in this case, all
fibers of $f$ lie on the quadric hypersurface through the point and the four lines, which is impossible.

\noindent In the last case, reasoning as above, up to rearranging indexes, we must have that:
$m_I=1$ for $I= \{1,2,3 \}, \{1,2,4 \}, \{1,2,5 \}, \{1,3,6 \}, \{1,4,6 \}, \{1,5,6 \}$, while
$m_I=0$ for any other $I$ with $|I|=3$. It follows from this that $m_{ \{1,3,4,5 \} }=2$, which is impossible.

\noindent We have to analyze the possibility that $d_i=3$ for all $i=1, \ldots, 6$. 
>From Lemma \ref{lem:conti} (i) and (iii) the unique numerical possibility for $f$ 
is that $m_I=1$ if $|I|=3$ and $m_I=0$ otherwise.
Let us consider the restriction of $f$ to the plane $\langle p_2,p_3,p_4 \rangle$. 
We have $m_{ \{1,2,3,4 \}}=0$ and from Lemma \ref{lem:restrizione} $f$ induces a fiber type morphism 
of relative dimension one on such a plane. We have an algebraic equivalence $C_1 \simeq D + L$ where $D$ is a general
fiber of such a restricted fibration; from the classification of fibrations
in $\Mo{4}$ and Stein's factorisation (see Corollary \ref{cor:stein}), since $d_1=3$, $D$ is either a conic or a set of $k \leq 3$ lines.
>From Lemma \ref{lem:restnum} the only possibility is that $D$ is a line or a couple of lines
in $\langle p_2,p_3,p_4 \rangle$ through  $\langle p_2,p_3,p_4 \rangle \cap \langle p_5,p_6 \rangle$. 
In either case the residual curve $L$, a conic or a line, 
intersects all other vital lines in $\P^3$ and this is impossible.

\bigskip
\noindent Let us now consider the case $n=7$.
>From Lemma \ref{lem:conti}[(iii)]
only two cases can occur for the genus $g$ of $F_f$: $g=1$ or $g=0$.
We will first exclude the case $g=1$.

\noindent If $g=1$, Lemma \ref{lem:conti} [(iii)] gives:
$$\sum_{|I|=2} m_I=0.$$
This implies that the base locus of $f$ consists of lines and planes only.
Since $n=7$ lines and planes are exchanged by standard Cremona transformations. 
Therefore, from Lemma \ref{lem:restrizione}, Lemma \ref{switch},
 Lemma \ref{lem:conti} (i), since $m_I=0$ for each $I$
such that $|I|=2$, $f$ induces a fiber type morphism of relative dimension one
 on each $E_I \simeq \overline{M}_{0,6}$.
 
\noindent Let us fix a Kapranov set $\{p_2, \ldots, p_7 \}$ in $\P^4$. Let $C_1=f_1(F_f)$ and 
$C \simeq D+L$ an algebraic equivalence, where $D$ is the general fiber of the restriction of $f$ to 
$\langle p_2,p_3,p_4,p_5 \rangle$.
Since we can work up to Cremona transformations, we may assume that $D$ is a set of $k$ lines,
for some $k \geq 1$. From the classification of fibrations on $\overline{M}_{0,6}$
and Stein's factorisation (see Corollary \ref{cor:stein}),
$D$ consists of a set of $k$ lines
through two fixed skew lines or through
a point, say $p_2$. If the latter is
true, then from Lemma \ref{lem:restnum} $m_{12}\neq 0$. Then we can
assume that all restrictions to vital
hyperplanes are of the first kind, or a
Cremona transform of these. This means
that the restricted base locus at any
vital hyperplane is either a pair
of lines or two incident lines and a
pair of
points. This is clearly impossible if
$m_{ij}=0$ for all possible values of
$i$ and $j$.

Assume that $F_f$ is a  rational curve
and thanks to Theorem \ref{th:linear} we
may assume that $f_1(F_f)$ is not a line.

>From Lemma \ref{lem:conti} [(iii)] we
have: 
\begin{equation}\sum_{|I|=2} m_I=6.
\label{eq:6}
\end{equation}

\noindent Suppose that for all $i=1,
\ldots, 7$ there are  at least two indexes
$j,k$ such that  $m_{\{i,j\}}, \ne 0, m_{\{i,k\}} \ne 0$. Then we must have
at least $7$ sets $I \subset \{1,\ldots,7 \}$ with $|I|=2$ and $m_I \ne 0$.
It then follows that we can always find an index $i \in \{1, \ldots,7 \}$
for which at most one number
$m_{\{i,j\}}$ is not zero.


Among all such indexes we will fix the one for which such a number
is minimum. With such a choice we have that, after renumbering, 
$$ \sum_{1 \in I, |I|=2} m_I =m_{\{ 1,2 \} } \leq 1.$$
Suppose indeed that $m_{\{1,2\}} \geq 2$; then for any index $i\ne 1,2$
either there exist $j,k \ne 1$  such that  $m_{\{i,j\}}, \ne 0, m_{\{i,k\}} \ne 0$
or there exists an index $j \ne 1$ such that $m_{\{i,j\}} \geq 2$, in any case violating
 $\sum_{1 \notin I, |I|=2} m_I=4$.

\noindent We will then assume that $ \sum_{1 \in I, |I|=2} m_I =m_{\{ 1,2 \} }=1$.
We fix a Kapranov set $\{p_2, \ldots, p_7 \}$ in $\P^4$ and  let $C_1=f_1(F_f)$.
>From Lemma \ref{lem:conti} (i) we have 
$$5d_1= 3+ 2\sum m_{\{1,j,l \}}+\sum m_{\{1,j,l,t\}} +2=5+5k,$$
for some $ k \geq 0$ for which $5k=2\sum m_{\{1,j,l\}}+\sum m_{\{ 1,j,l, t\}}$, so that $d_1=k+1$.
There are ten hyperplanes containing
$p_2$. Hence there are at least 4
hyperplanes $V^1_J$ with $m_J=0$. Let us
fix one of those, say
$\langle p_2,p_3,p_4,p_5 \rangle$,  such that $m_{\{ 1,2,3,4,5 \}}=0$.
The morphism $f$ is not linear and
$m_{\{1,2\}}=1$, then we can assume that 
the base locus of $f$ in $\langle p_2,p_3,p_4,p_5 \rangle$ does not
consist only of the point $p_2$. We will show that $f$ induces 
a fibration by curves on $E_{\{ 1,2,3,4,5 \}}$. If there
exists a vital plane $V^1_J$ with $J \subset \{ 1,2,3,4,5 \}$ and $m_J \ne 0$
we are done thanks to Lemma \ref{lem:restrizione}. If not, there will exists
a line $V^1_J$ with $J \subset \{ 1,2,3,4,5 \}$ and $m_J \ne 0$; but in this case
we can produce a Cremona transformation which transforms the line in a plane
and we can still apply Lemma \ref{lem:restrizione}.

\noindent Let us write $C_1 \simeq D+L$ for the algebraic equivalence of $C_1$ with
the union of $D$, the general fiber of the restricted fibration, and a residual $L$.
The curve $D$ contains $p_2$ and
$m_{12}=1$. In particular, from Lemma \ref{lem:restnum} $D$ is smooth
at $p_2$. Then from the classification
of fibrations by curves on
$\overline{M}_{0,6}$, and
Lemma \ref{lem:restnum}, 
$D$ is either a line through $p_2$
or a conic through $p_2$, another point, and two
intersecting lines 
(the Cremona transform of a line through
two skew lines). 

The degree of $C_1$ is $k+1$ and the only other relevant intersection
of $D$ (see definition \ref{def:relev}) is with the plane $\langle p_3, p_4, p_5 \rangle$,
hence we must have $m_{ \{1,3,4,5 \}}=k$. 

If $D$ is a line through $p_2$, $C_1$ is a conic through $p_2$, which intersects the plane
$\langle p_3, p_4, p_5 \rangle$ and either two lines, or a line and two planes
or four planes. In any case, since $m_{\{1,3,4,6,7 \}}=m_{\{1,4,5,6,7 \}}=m_{\{1,3,5,6,7 \}}=1$
we will always be able to find at least other three hyperplanes which are intersected
by $C_1$ at a general point, contradicting $\sum_{|I|=2} m_I=6.$

\noindent We must analyze the possibility that $D$ is a conic
through $p_2$, another point and two concurrent lines in $\langle p_2,p_3,p_4,p_5 \rangle$.
>From Lemma \ref{lem:restnum} and our choice of the index, the only possibility is that
the other point is $\langle p_2,p_3,p_4,p_5 \rangle \cap \langle p_6, p_7 \rangle.$

Since $C_1$ cannot intersect too many hyperplanes,
we must have that $C_1=D$ is a conic. Since $C_1$ is a conic  $k=1$, so that
$5=2\sum m_{\{1,j,l\}}+\sum m_{\{ 1,j,l, t\}}$; moreover $C$ intersects
the lines $\langle p_3, p_5 \rangle$ and $\langle p_4,p_5 \rangle$
and the plane $\langle p_2, p_3, p_4 \rangle$.

\noindent The unique possibility is then that 
$m_{\{1,2,3,4 \}}=m_{\{1,3,5,6\}}=m_{\{1,4,5,6\}}=m_{\{1,6,7,\}}=1$
are the only non zero intersections of $C$ with vital spaces.
It turns out that fibers and base locus of $f$ coincide with the one given by

$$\phi_{\{1,5,6 \}} \times \phi_{\{2,3,4\}} \times \phi_{\{2,6,7 \}}: \overline {M_{0,7}} \to \overline {M_{0,4}} \times \overline {M_{0,4}} \times \overline {M_{0,4}}$$
which has equidimensional fibres, so that $f$ factors through such a forgetful map.

The case for which $m=0$ is treated similarly; in this case the base locus consists only
of lines and planes, we find many hyperplanes which are fibred by $f$ and the only
possibility is that we end up with a fibration given, via $f_1$, by conics through the lines
$\langle p_5, p_6 \rangle$, $\langle p_2, p_4 \rangle$, $\langle p_2, p_3 \rangle$, $\langle p_6, p_7 \rangle$.
This is the above fibration up to switching indexes $1$ and $7$.
\end{proof}

Applying Stein's factorisation we obtain the following:

\begin{corollary} \label{cor:stein}
Let $n \leq 7$, let $\{1, \ldots, n \}= \{i_1, \ldots, i_n \}$ and let
$f$ be a fiber type morphism of relative dimension one. Then there exists a factorisation
$f=g \cdot h$ where $g$ is a generically finite morphism and $h$ is one of the following:
\begin{itemize}
\item[(i)] $h=\phi_{i_j}$ for some $j \in \{1,\ldots,n\}$,
\item[(ii)] $h= \phi_{\{i_1,i_2\}} \times \phi_{\{i_3,i_4\}}$ and $n=6$,
\item[(iii)] $h= \phi_{\{i_1,i_2\}} \times \phi_{\{i_3,i_4,i_5\}}$ and $n=7$,
\item[(iv)] $h= \phi_{\{i_1, i_2, i_3\}} \times \phi_{\{i_4,i_5,i_6\}} \times \phi_{\{i_1,i_4, i_7\}}$ and $n=7$.
\end{itemize}
\end{corollary}

\section{Nonlinear fibrations}
\label{sec:nonl}
Case (iv) in Proposition \ref{prop:curves} is an example of a fibration by curves
whose fibers cannot be reduced by standard Cremona transformations centered at subsets
of the Kapraneov set to lines. In this section we will show that there exist fibrations
by curves whose fibers are not even rational. Nevertheless the examples we find are 
forgetful morphisms.

\noindent The first step is to
generalise Castravet and Televev
Criterion, Proposition \ref{ct} see \cite{CT} 
to recognise when a forgetful morphism
is surjective when all indexes are
forgotten. 

We start with a collection $I_1, \ldots , I_k$ of subsets of $\{1, \ldots, n \}$,
with no inclusions among any two sets of the collection. In Castravet-Televev
comtext there exists an index, say $1$, such that 
$1 \notin \cup_{j=1,\ldots,k} I_j$.
If we consider Kapranov's map $f_1$ we have $k$ linear spaces $U_i:=V^1_{I_i}$.
Now, the condition that for any $ S \subset \{1, \ldots, k \}$,
$$(*) \qquad n-|\cap_{j \in S}I_j|-3 \geq \sum_{j \in S}(n-|I_j|-3)$$
is equivalent to the condition that for any $ S \subset \{1, \ldots, k \}$
$$(+) \qquad {\rm codim} (\cap_{j \in S} U_j) -1 \geq \sum_{j \in S} [{\rm codim}(U_i)-1].$$

Finally, condition (*)  and $n-3-\sum_{j=1}^k(n-|I_j|-3)=:h \geq 0$ ensure surjectivity of $\psi= \prod_{j=1}^k \phi_{I_j}$ because
condition (+) ensures maximal rank for the differential of $\prod_{j=1}^k \pi_{U_j}$ and $h \geq 0$ gives surjectivity.

\noindent We will assume now to have a collection $I_1, \ldots , I_k$ of subsets of $\{1, \ldots, n \}$,
with no inclusions among any two sets of the collection, and that for any $j \in \{1, \ldots, n \}$ there exists
$ l \in \{ 1, \ldots, k \} $such that $j \in I_l$. We will show that the same condition (*) is necessary and sufficient
for surjectivity of the forgetful morphism $\psi=\prod_{j=1}^k \phi_{I_j}$. Suppose that $1 \in \cap_{j=1, \ldots, t} I_j$ and
that $1 \notin \cup_{i \geq t+1} I_j$, with $0\leq t\leq k$. Consider the Kapranov map $f_1 : \Mo{n} \to \P^{n-3}$ and consider the following
subspaces of $\P^{n-3}$: if $j \leq t$, $U_j:=V^1_{I_j\setminus\{1\}}$, if $j>t$, $U_j:V^1_{I_j}$.

\begin{proposition}
With the above notations, suppose that
for any  $S \subseteq \{1, \ldots k \}$ we have:
$ (*) \qquad n-|\cap_{j \in S}I_j|-3 \geq \sum_{j \in S}(n-|I_j|-3)$.
Then the differential of the projection $\prod_{j=1}^k \pi_{U_j}$ has maximal
rank. In particular, if $h \geq t$, $\psi$ is surjective. If $h < t$
there exists a general point $x \in \P^{n-3}$
such that if $U^{'}_j:= \langle U_j , x \rangle$ for $j \leq t$
and $U_j^{'}:= U_j$ for $j>t$, the linear projection $\prod_{j=1}^{k} \pi_{U^{'}_j}$
is surjective.
\label{presurj}
\end{proposition}
\begin{proof}

The morphism $\psi$ can be factorised through
 $\psi_{1}:= \phi_{I_1 \setminus \{1\}} \times \cdots
 \phi_{I_t \setminus \{1 \}} \times \cdots \times \phi_{I_k}$
followed by $\phi_{1}^{t} \times {\rm id}^{k-t}$,
where $\phi_{1}^{t}=\phi_1 \times \cdots \phi_1$ $t$ times.
The map $\psi_1$ has the property of being factorisable through the linear projection
\[
\xymatrix{
\overline{M}_{0,n}\ar[d]^{\psi_1}\ar[r]^{f_1}
&\P^{n-3} \ar@{.>}[d]^{\pi_{U_1} \times \cdots \times \pi_{U_k}}\\
\overline{M}_{0,n-|I_1|+1} \times \cdots \times \overline{M}_{0,n-|I_k|}\ar[r]^{f_1}
&\P^{n-3-|I_1|+1} \times \cdots \times \P^{n-3-|I_k|}}
\]

Now, if $I_1, \ldots, I_K$ satisfy (*) a-fortiori (*) is satisfied
for the collection $I_1 \setminus \{1\}, \ldots, I_t \setminus \{1\}, I_{t+1}, \ldots, I_k$.
In particular the differential of $\pi_{U_1} \times \cdots \times \pi_{U_k}$ has maximal rank.
If $h \geq t$, the projection $\pi_{U_1} \times \cdots \times \pi_{U_k}$ is surjective 
and $\psi= (\phi_{1}^{t} \times {\rm id}^{k-t}) \circ \psi_1$ is the composition of 
two surjections, so that it is surjective.

In case $h<t$, condition (*) ensures that if $x \in \P^{n-3}$
is in linearly general position with respect to all other  points
of the Kapranov set $K$, the linear spaces $U^{'}_j:= \langle U_j , x \rangle$ for $j \leq t$
and $U_j^{'}:= U_j$ for $j>t$ satisfy condition (+) and are then
such that that $\pi_{U^{'}_1} \times \cdots \times \pi_{U^{'}_k}$
has maximal rank differential. Since
the expected dimension of a fibre of $\prod_{j=1}^{k} \pi_{U^{'}_j}$
is $h \geq 0$, the linear projection $\prod_{j=1}^{k} \pi_{U^{'}_j}$
is surjective.
\end{proof}
As a corollary we have the following:

\begin{theorem}
The morphism $\psi=\prod_{j=1}^k \phi_{I_j}$ is surjective if and
only if for any $S \subseteq \{1, \ldots k \}$ we have:
$$ (*) \qquad n-|\cap_{j \in S}I_j|-3 \geq \sum_{j \in S}(n-|I_j|-3).$$
\label{surj}
\end{theorem}
\begin{proof}
Suppose that $\psi$ is surjective and
choose $S \subset \{ 1, \ldots, k \}$; let $W= \cap_{j \in S}I_j$ and let $w=|W|$.
The product of maps $\phi_{I_j}$ with $j \in S$ is also surjective and factors
through the forgetful $\phi_W : \Mo{n} \to \overline{M}_{0,n-w}$, so that also
the forgetful map $ \prod_{j \in S} \phi_{I_j \setminus W} :
\overline{M}_{0,n-w} \to \prod_{j \in S} \overline{M}_{0,n-|I_j|}$
is surjective. In particular the dimension of the
source must be bigger than the dimension of the target
so that $ n-w-3=n-|\cap_{j \in S}I_j|-3 \geq \sum_{j \in S}(n-|I_j|-3).$

We need to show that the condition is also sufficient.
If $t=0$ this is Castravet-Televev Theorem 3.1 in \cite{CT}; if $t=k$
then we can factor $\psi$ via $\phi_1$ and since
$(*)$ holds the procedure will terminate,
so that we will also assume $0 <t <k$. In view of Proposition \ref{presurj}
we will assume $h < t$.

As already noticed in Proposition \ref{presurj}
the morphism $\psi$ can be factorised through
 $\psi_{1}:= \phi_{I_1 \setminus \{1\}} \times \cdots
 \phi_{I_t \setminus \{1 \}} \times \cdots \times \phi_{I_k}$
followed by $\phi_{1}^{t} \times {\rm id}^{k-t}$.
The map $\psi_1$ has the property of being factorisable through the linear projection
\[
\xymatrix{
\overline{M}_{0,n}\ar[d]^{\psi_1}\ar[r]^{f_1}
&\P^{n-3} \ar@{.>}[d]^{\pi_{U_1} \times \cdots \times \pi_{U_k}}\\
\overline{M}_{0,n-|I_1|+1} \times \cdots \times \overline{M}_{0,n-|I_k|}\ar[r]^{f_1}
&\P^{n-3-|I_1|+1} \times \cdots \times \P^{n-3-|I_k|}}
\]

where we recall that if $j \leq t$, $U_j:=V^1_{I_j\setminus\{1\}}$,
while if $j>t$, $U_j:V^1_{I_j}$.

From now on we will mainly work in the projective setting
so that for instance instead of considering a map like
$\phi_1$ we will mainly be interested in its projective
counterpart $f_1 \circ \phi_1 \circ (f_1)^{-1}$, which is a
rational fibration whose general fiber is a rational normal curve 
containing all the points of a Kapranov set.

We apply Proposition \ref{presurj} and we
find a general point $x \in \P^{n-3} $  such that
$\pi:= \prod_{j=1}^{k} \pi_{U^{'}_j}$ is surjective. Notice that $\pi$ is obtained by composition
of $\pi_{U_1} \times \cdots \times \pi_{U_k}$ and $t$ linear projections from a point on
the first $t$ factors of the image.
Condition $(+)$ is equivalent to a finite number of linear equations
with maximum rank. We can then find a general
subspace $U^{'}_{k+1}$ of codimension
$h+1$ in such a way that codimension $h$ linear subspaces containing
$U^{'}_{k+1}$ are transverse to fibers of $f_1 \circ \psi \circ f_1^{-1}$, so that $\psi$ is surjective if
and only if $(f_1 \circ  \psi \circ f_1^{-1}) \times \pi_{U^{'}_{k+1}}$ is generically finite.
We will show that $(f_1 \circ \psi \circ f_1^{-1}) \times \pi_{U^{'}_{k+1}}$ is generically finite.
In order to do this it will be enough to show that fibers of $\pi_{U_1} \times \cdots \times \pi_{U_k} \times \pi_{U^{'}_{k+1}}$ are transverse to fibers of $(f_1 \circ \phi_1 \circ f_1^{-1})^t$.

We may assume that $\pi_{U_1} \times \cdots \times \pi_{U_k} \times \pi_{U^{'}_{k+1}}$ is birational, 
by generality of $U^{'}_{k+1}$;
this implies that for general $u \in
\prod_{j >t} \P^{n-3-|I_j|}$, $v \in
\P^{h}$, the general fiber of $\pi_{U_1} \times \cdots \times \pi_{U_k} \times \pi_{U^{'}_{k+1}}$,

$$H_{u,v}:=(\pi_{U_1} \times \cdots \times \pi_{U_k} \times \pi_{U^{'}_{k+1}})^{-1}
( \prod_{j=1}^{t} \P^{n-3- |I_j|+1} \times \{(u,v)\})
\subset \prod_{j=1}^{t} \P^{n-3-
  |I_j|+1}$$
 has codimension $t$ and
is transverse to fibers of linear projections
from the images of $x$ on each of the first $t$ factors. This implies that
$H_{u,v}$ is birational to the complete intersection
of pullbacks of $t$ hyperplanes on the factors. It then follows that $H_{u,v}$,
for general $u,v$
does not contain all coordinate points in each
factor. Suppose that $\psi$ is not surjective,
so that $(f_1 \circ \psi_1 \circ f_1^{-1}) \times \pi_{U^{'}_{k+1}}$ is not birational.
It then follows that $H_{u,v}$ is not transverse to the fibers of $(f_1 \circ \phi_1 \circ f_1^{-1})^t$ so that
if $(s,u,v) \in {\rm Im}((f_1 \circ \psi \circ f_1^{-1}) \times \pi_{U^{'}_{k+1}})$
$((f_1 \circ \phi_1 \circ f_1^{-1})^t)^{-1}(s,u,v) \subset H_{u,v}$ must contain at least on a factor a rational
normal curve through all coordinate points, i.e. a fiber
of $f_1 \circ \phi_1 \circ f_1^{-1}$. This contradicts the property of $H_{u,v}$ of
not containing all the points of a Kapranov set in a factor.
\end{proof}

We will apply this criterion to ensure the surjectivity of some forgetful morphisms.
The basic example is the following, that is constructed by the consideration of linear series
on some genus $5$ hypergraph curves, as in \cite {CT} 

\begin{example}\label{ex:elliptic} Let
  us consider the following forgetful map,
$$\psi =\phi_{\{1,2,3,4 \}} \times \phi_{\{1,2,5,6\}} \times \phi_{\{3,4,7,8\}} \times \phi_{\{5,6,7,8\}}: \overline{M}_{0,8}
\to \prod \overline{M}_{0,4}.$$
\end{example}

\begin{proposition}
The morphism $\psi$ is surjective and $F_{\psi}$ is a smooth elliptic curve.
\end{proposition}
\begin{proof}
The morphism $\psi$ is dominant because the sets of indexes satisfy condition $(*)$
of Theorem \ref{surj}. Consider $f_1$, the associated Kapranov
subset $K= \{ p_2, \ldots, p_8 \}$
and the diagram:
\[
\xymatrix{
\overline{M}_{0,8}\ar[d]^{\psi}\ar[r]^{f_1}
&\P^{5} \ar@{.>}[d]^{\pi_1 \times \pi_2  \times \pi_{V^1_{\{1,3,4,7,8 \}} }\times \pi_{V^1_{\{1,5,6,7,8 \}}}} \\
\P^1 \times\P^1\times \P^1 \times \P^1\ar[r]^{=}&
\P^1 \times \P^1 \times\P^1\times \P^1}
\]
where $\pi_1$ and $\pi_2$ are compositions of a Cremona transformation and a linear projection.
After Lemma \ref{lem:fibers} we have that: 
\begin{itemize}
\item[.] fibers of $\pi_1$ are quadric
  cones with vertex $\langle p_2, p_3,
  p_4 \rangle$, and passing through
  $\{p_5,p_6,p_7,p_8\}$; 
\item[.] fibers of $\pi_2$ are quadric
cones with vertex $\langle p_2, p_5, p_6
\rangle$, and passing through
  $\{p_3,p_4,p_7,p_8\}$;
\item[.] fibers of
$\pi_{V^1_{\{3,4,7,8\}}} \times
\pi_{V^1_{\{5,6,7,8 \}}}$ 
are linear spaces of dimension three
containing the line $\langle p_7,
p_8\rangle$ and meeting the disjoint
lines $\langle p_3, p_4 \rangle$ 
and $\langle p_5, p_6 \rangle$.
\end{itemize}
In particular the general fiber $C_1=f_1(F_{\psi})$ is
 a complete intersection of
two quadric cones, say $Q_1$
and $Q_2$, in a $\P^3$, say $H_C$. Its arithmetic genus is $1$ so that if smooth is connected.   Let us check the smoothness
of $C_1$.  The quadric $Q_1$ is singular
at the point $v_1:=C_1\cap \langle
p_3,p_4\rangle$, while $Q_2$ is 
singular at $v_2:=C_1\cap\langle p_5, p_6
\rangle$.  Then the only possible
singularities of $C_1$ are at the points
$\{p_7,p_8\}$. 
To conclude observe that by construction, for $i=1,2$ and $j=7,8$, we have
$$\T_{p_j}Q_i\supset\langle v_i,p_j\rangle$$
hence
$$\dim\T_{p_j}C_1=1.$$
\end{proof}
The map in
the example can be factored as follows,
$$ \psi =((\phi_{\{3,4\}} \times
\phi_{\{5,6\}}) \times (\phi_{\{3,4\}}
\times \phi_{\{5,6\}})) \circ
(\phi_{\{1,2\}} \times \phi_{\{7,8\}})$$ 
where the first map is an immersion into a divisor which is the pullback of a quadric in $\P^3 \times \P^3$
and the second one is the product of two
fibrations whose fibers are pullbacks of
lines. This description  shows  that $\psi$ has 
equidimensional fibers. The morphism $\psi$ can be
realized, following \cite{CT}, in terms
of limit linear series 
on hypergraph curves.
In this dictionary Example \ref{ex:elliptic}
corresponds to a component of the locus
of limit pencils of degree $4$  on a  
stable hypergraph curve of arithmetic
genus $5$ obtained from $4$ smooth
rational curves by gluing at
complementary  indexes of the collection
$I_1,I_2,I_3,I_4$. We do not know how to
generalize this point of view  to
stable hypergraph curves of higher genus
with a one-dimensional 
component of special linear
series. On the other hand the extension
is at hand for  the description in
Example \ref{ex:elliptic}. 

\begin{example} \label{ex2}
Let $K= \{p_2, \ldots, p_{10}
\}\subset\P^{7}$ be a Kapranov
set. Consider the map 
$\psi'=\psi\circ
\phi_{\{9,10\}}$. The general fiber is a
cone over a general fiber of $\psi$, with vertex $\langle
p_9,p_{10}\rangle$. To produce a
fibration by curves we then slice twice
as follows. Fix a forgetful map
centered on a codimension
$2$ vital cycle transverse to $\langle
p_9, p_{10} \rangle $, for instance
$\phi_{\{3,4,5,6,7,8 \}} $  
and a pencil of quadrics with vertex a
Kapranov subset of dimension $4$, for
instance $\phi_{\{1,2,5,6,7,8\}}$. This
induces a forgetful map
$$\psi_{10}:\overline{M}_{0,10}\to
\prod_1^6\P^1$$
whose general fiber is a  genus $5$
smooth canonical curve, obtained as
complete intersections of three quadrics
and three hyperplanes. 
The morphism described is then given by
the following collection of forgotten sets:
\begin{eqnarray*}\{1,2,3,4,9,10\},
  \{1,2,5,6,9,10 \}, \{3,4,7,8,9,10 \},\\
  \{5,6,7,8,9,10 \}, \{1,2,5,6,7,8\}
  , \{3,4,5,6,7,8 \}
\end{eqnarray*}
Observe that every index is repeated at least
three times and condition $(*)$ is satisfied. 

One can go further to produce fibrations
on $\overline{M}_{0,2k}$, for each $k
\geq 4$, whose general fibers are smooth complete intersections of $k-2$
quadrics. 
Let us briefly sketch this construction.
Start with the above example for $k=4$ and
suppose we have constructed such a fibration on $\overline{M}_{0,2k}$, with
sets of indexes $I_1, \ldots, I_{2k-4}$ of cardinality $2k-4$ each, we add to each $I_j$ the new indexes $2k+1, 2k+2$ and we form
$I_{2k-3}, I_{2k-2}$ by $I_{2k-3}= \{ 1,2, 5,6, \ldots,2k-1,2k \}$,$I_{2k-2}= \{ 3,4, 5,6, \ldots,2k-1,2k \}$.  
Condition (*) is verified for this set of indexes so that the forgetful map 
$$\psi_{2k+2}:
\overline{M}_{0,2k+2} \to \prod \P^{1}$$
is surjective. The general fiber, via $f_1$, is a
complete intersection of hyperplanes and
$(k-3)$ hyperquadrics.  It is easy to recover
smoothness and connectedness of fibers
as before.
Therefore the  genus of a general fiber is 
$$g(F_{\psi_{2k+2}})=(k-4)2^{k-3}+1$$
\end{example}

\section{Fibrations with low-dimensional target}
\label{sec:tar}

As a final application we want to give
some evidence to the following conjecture.
\begin{conjecture}
\label{conj:fac} Every fibration from 
$\Mo{n}$ factors
through a forgetful map. 
\end{conjecture}
This is the best result we are able to
prove in this direction. 

\begin{theorem} Let  $f: \Mo{n} \to X$
  be a dominant morphism with ${\rm
    dim}X < n-3$.
If either $\dim X\leq 2$ or  $n \leq 7$
then  $f=g \circ h$, where $g$ is a generically finite
morphism and $h$ is a forgetful
morphism. More precisely  if  $n \leq 7$ then $h$ is
one of the following  morphisms:
\begin{itemize}
\item[(i)] $\phi_I : \Mo{n} \to \overline{M}_{0,n-|I|}$, with $|I| \leq n-4$,
\item[(ii)] $\phi_I \times \phi_J :\overline{M}_{0,n} \to
\overline{M}_{0,n-|I|} \times \overline{M}_{0,n-|J|}$, where $|I|+|J| \leq 3$ and if $|I \cap J| \geq 1$ $n=7$ and $|I \cap J| =1$
\item[(iii)] $ \phi_{\{i_,i_2,i_3\}} \times \phi_{\{i_4,i_5,i_6\}} \times \phi_{\{i_1,i_4,i_7\}}:
\overline{M}_{0,7} \to \overline{M}_{0,4} \times \overline{M}_{0,4} \times
\overline{M}_{0,4}$.
\end{itemize}\label{th:fac7}
\end{theorem}
\begin{proof}
By Stein factorization we may assume
that $f$ has connected fibers. 
The
theorem has been proved in \cite{BM}, Corollary 3.9 for
fibrations over a curve. Keep in mind
that $\Mo{n}$ is rational and therefore
if $\dim X=1$ then $X\iso\P^1$.
Thanks to Corollary \ref{cor:stein}
the only  remaining case is that of
$\dim X=2$. We will prove it by
induction on $n$. The first step, that
is the case of $\Mo{6}$, is already
proved in Proposition \ref{prop:curves}.
Consider such a fibration, and the usual diagram
\[
\xymatrix{
\overline{M}_{0,n}\ar[d]^{f}\ar[r]^{f_1}
&\P^{n-3} \ar@{.>}[d]^{\pi} \\
X\ar[r]^{=}&X}
\]

Let $G:=f_1(F_f)$ be the image of a general
fiber.
\begin{claim} There is a 
vital hyperplane intersecting $G$ in a
general point.
\end{claim}
\begin{proof}[Proof of the Claim] In any case the 
general fiber has to intersect in a
general point at least one of the
$E_I$. If $|I|=2$ we are done. Assume
that $|I|\geq 3$, $\{2,3\}\in I$ and let
$\Gamma=f_{1*}(F_f\cap E_I)$ be the
intersection curve. Consider the divisor
$E_I\iso \Mo{|I|+1}\times\Mo{|I^*|+1}$ with
natural projections $\pi_I$ and $\pi_{I^*}$
onto the factors. 
We may assume that $\Gamma$ is not contained
in a fiber of $\pi_I$. Hence we have
$$E_{2,3}\cap
\Gamma\neq\emptyset.$$
\end{proof}

Assume that such a vital hyperplane is
$H:=H^{1\vee}_{2,3}$. Then the restriction 
$$\f:=f_{|H}:E_{2,3}\iso\Mo{n-1}\to X$$
is of fiber type and, by induction hypothesis, we
can factor $\f$
via a map, say $\psi$, onto either
\begin{itemize}
\item[.] $\Mo{5}$ or
\item[.] $\Mo{4}\times\Mo{4}$.
\end{itemize}
 Let
$\L_1=f_{1*}(f^*A)$, for some very ample
linear system $A$ on $X$. 
We cannot conclude by induction because
the morphism  $\f$ could have non
connected fibers.  
Let $\L_{1|E_{6,7}}=\M+F$ for some movable
linear system $\M$ and fixed component
$F$. We aim to study the fixed component
$F$ to show that it is empty. 

With the description of $\f$ it
is easy to produce many vital
hyperplanes intersecting the general
fiber of $\f$. Then changing the vital divisor
$E_{2,3}$ and repeating the argument we
can see, as in the proof of Theorem 3.7 in \cite{BM}, that $F$ is empty.

Next we study the base locus $\Bl\L_1$. Again changing
the divisor $H$  we see that $\Bl\L_1$ is either a
codimension 3 linear space or a pair of
codimension two linear spaces. In both cases
let $Y$ be the blow up of $\P^{n-3}$ along
the base locus. 

If $\Bl\L_1$ is irreducible then $Y=\P(\O(1)\oplus\O)$ 
and there is a unique fiber type
morphism onto a surface.
Assume that $\Bl\L_1=V_1\cup V_2$ has two
irreducible components and $\cod (V_1\cap
V_2)=2$. Let $Y$ be the
blow up of $\P^{n-3}$ along
$\Bl\L_1$.  Let 
$$B_1\iso\P_{\P^{n-6}}(\O(1)\oplus\O)\times\P^1\ {\rm
  and}
\ B_2\iso\P_{\P^{n-6}}(\O(1)\oplus\O)\times\P^1$$ be the
exceptional divisors in $Y$ and $b_1$,
 $b_2$ the class of exceptional 
curves (i.e. the fibers of the
projection onto the first factor) in $B_1$, respectively $B_2$. Let
$m$ be the transform of a general line
in $\P^{n-3}$ and $l$ the transform of a
 line intersecting both $V_1$ and
$V_2$.
Then as vector spaces we have
$$N_1(Y)_\R=\langle b_1,b_2,m\rangle $$
and
$$N^1(Y)_\R=\langle B_1,B_2,H\rangle $$
where $H$ is the pullback on $Y$ of a
hyperplane. The numerical class of a
curve $\Gamma$  is therefore uniquely determined
by the triple 
$$(B_1\cdot\Gamma,B_2\cdot\Gamma,H\cdot\Gamma)$$ 
In this notation we have
$$b_1\equiv(-1,0,0),\ b_2\equiv(0,-1,0),\ l\equiv(1,1,1).$$ 
Moreover for any irreducible curve
$Z\subset Y$ we have $B_i\cdot Z\leq
H\cdot Z$. Hence any effective curve is numerically
proportional to a linear combination of
these three curves, that is
$$NE(Y)=\langle b_1,b_2,l\rangle. $$
Note that $Y$ is Fano and therefore 
$NE(Y)=\overline{NE}(Y)$. Hence on $Y$
there is a unique morphism 
of fiber type onto a surface.

The final case, that  is  $\Bl\L_1=V_1\cup V_2$
and $\cod V_1\cap V_2=1$, can be treated
similarly. This time we may assume that 
$$B_1\iso\P^{n-5}\times\P^1$$ 
and $B_2$ is the blow up of
$\P^{n-5}\times\P^1$ along a codimension
2 smooth section
$\Sigma\iso\P^{n-6}$. Then the three
curves 
that generate $\overline{NE}(Y)$ are
$$b_1\equiv(-1,0,0),\ b_2\equiv(1,-1,0),\ l\equiv(1,1,1).$$
Again there is a unique morphism of
fiber type onto a surface.
This is enough to conclude.
\end{proof}

\end{document}